\documentclass[a4paper,12pt]{article}
\usepackage[centertags]{amsmath}
\usepackage{amsfonts}
\usepackage{amssymb}
\usepackage{amsthm}
\usepackage{dsfont}
\usepackage{tikz, subfigure}
\usepackage{verbatim}
\usepackage{color}

\newtheorem{Theorem}{Theorem}[section]
\newtheorem{Proposition}{Proposition}[section]
\newtheorem{Lemma}{Lemma}[section]

\newtheorem{Definition}{Definition}[section]

\AtEndDocument{\bigskip{\footnotesize%
  \textsc{Department of Mathematics, Shanghai Key Laboratory of PMMP, East China Normal University, Shanghai 200241,
People's Republic of China} \par
  \textit{E-mail address:} \texttt{	
243564045@qq.com.} \par
}}
\AtEndDocument{\bigskip{\footnotesize%
  \textsc{Department of Mathematics, Ningbo University, Ningbo, Zhejiang, People's Republic of China} \par
  \textit{E-mail address:} \texttt{kanjiangbunnik@yahoo.com,jiangkan@nbu.edu.cn} \par
}}
\AtEndDocument{\bigskip{\footnotesize%
  \textsc{Department of Mathematics, Shanghai Key Laboratory of PMMP, East China Normal University, Shanghai 200241,
People's Republic of China} \par
  \textit{E-mail address:} \texttt{wxli@math.ecnu.edu.cn.} \par
}}

\begin{document}

\title{xxxx}
\date{}
 \title{Visibility of Cartesian products of  Cantor sets}
\author{Tingyu Zhang\thanks{Tingyu Zhang is  the corresponding author}, Kan Jiang  and Wenxia Li}
\maketitle{}
\begin{abstract}
Let $K_{\lambda}$ be the attractor of the following IFS
\begin{equation*}
\{f_1(x)=\lambda x, f_2(x)=\lambda x+1-\lambda\}, \;\;0<\lambda<1/2.
\end{equation*}
Given $\alpha \geq  0$, we say the line $y=\alpha x$ is visible through $K_{\lambda}\times K_{\lambda}$ if
$$
\{(x, \alpha x): x\in \mathbb R\setminus \{0\}\}\cap ((K_{\lambda}\times K_{\lambda}))=\emptyset.
$$
Let $V=\left \{\alpha \geq 0: y=\alpha x \mbox{ is visible through } K_{\lambda}\times K_{\lambda} \right \}$.
In this paper, we give a completed description of $V$, e.g.,  its Hausdoff dimension and its topological property. Moreover, we also discuss another type of visible problem which is related to the slicing problems. 
\end{abstract}
\section{Introduction}
Projections, sections,  geodesic curves and  visiblity are the main problems in geometry measure theory. It  is related to many aspects of fractal geometry, for instance,  the arithmetic sum of two self-similar sets is indeed  the projectional problem \cite{Hochman2012,PS}; sections of some fractal sets are connected to the multiple representations of real numbers \cite{XiKan2}; geodesic curves on fractal sets are distinct from the classical differential manifolds \cite{LX}.  For more results on these problems see \cite{wenxi,XWX,Barany,XiC,XiD,XiF,XiG} and references therein. In this paper, we shall consider the visiblity of  the Cartesian products of  some Cantor sets.

Given $\alpha\geq  0$ and some subset $F \subset \mathbb{R}^2$, we say the line $y=\alpha x$ is visible through $F$ if
$$\{(x, \alpha x):  x\in \mathbb R\setminus \{0\}\}\cap F=\emptyset.$$
The concept of  ``visiblity" was investigated by many scholars.  Nikodym \cite{Nikodym} constructed a subset $F$ of $\mathbb{R}^2$ such that   every  point of $F$ is visible from two diametrically opposite directions. In convex geometry, Krasnosel \cite{Helly} offered a beautiful criterion which enables us to check whether the entire boundary of a compact set of $\mathbb{R}^2$ is visible from an interior point.
 Falconer and  Fraser \cite{FF} proved that  for a class of plane self-similar sets  when the attractor $F$ has Hausdorff dimension greater than $1$ then the Hausdorff dimension of the visible subset is $1$. The readers can find more related results in \cite{Neil,JJ1,JJ2,Bond}.

In this paper, we shall analyze   the following   self-similar set.
Let $K_{\lambda}$ be the attractor with the IFS
\begin{equation*}
\{f_1(x)=\lambda x, f_2(x)=\lambda x+1-\lambda\}, \;\;0<\lambda<1/2,
\end{equation*}
i.e.,
\begin{equation}\label{klambda }
K_\lambda =f_1\left (K_\lambda\right )\cup f_2\left (K_\lambda\right ).
\end{equation}
Let
$$V=\{\alpha\geq 0: y=\alpha x \mbox{ is visible through } K_{\lambda}\times K_{\lambda} \}.
$$
It is easy to verify  that the line $y=\alpha x$ is visible
through $K_{\lambda}\times K_{\lambda}$ if  and only if
$$\alpha\notin \frac{K_{\lambda }}{
K_{\lambda }\setminus \{0\}}:=\left\{\dfrac{x}{y}:x,y\in K_{\lambda }, y\neq 0\right\}.$$
Thus,
\begin{equation}\label{setV}
V=[0, +\infty )\setminus \frac{K_{\lambda }}{
K_{\lambda }\setminus \{0\}}.
\end{equation}

By $A^o$ we denote the set of interior points of $A$, by $m(A)$ we denote the Lebsgue measure of $A$. In this paper, we obtain the following results.
\begin{Theorem}\label{Main}
Let $K_{\lambda}$ be given by (\ref{klambda }).
Then
\begin{itemize}
\item [(1)] When $\dfrac{3-\sqrt{5}}{2}\leq \lambda<1/2$,   $V=\emptyset$;

\item [(2)] When $ \dfrac{1}{3}\leq \lambda<\dfrac{3-\sqrt{5}}{2}$,
$$
V=(0, +\infty )\setminus \bigcup_{k=-\infty}^{\infty}\lambda^k\left[1-\lambda, \dfrac{1}{1-\lambda}\right].
$$
\item [(3)] When $ 0< \lambda<\dfrac{1}{3}$,  $V^o\ne \emptyset$. In particular,
when $\dfrac{1}{4}< \lambda<\dfrac{1}{3}$,  $([0,+\infty )\setminus V)^o\ne \emptyset $;
when $0< \lambda\leq \dfrac{1}{4}$, $m([0,+\infty )\setminus V)=0$ and $\dim _H([0,+\infty )\setminus V)=\dfrac{\log 4}{-\log \lambda}$.
\end{itemize}
\end{Theorem}
There are mainly two types of visible problem \cite{FF}. Now, we shall consider another one.  First, we introduce some definitions. 
Let $l_\theta$ denote the  line going through the  origin in direction $\theta\in(0,\pi/2), $ that is, 
$$l_{\theta}=\{(x,(\tan\theta) x): x\in \mathbb{R}\}.$$
Given $\theta\in (0,\pi/2)$. 
The visible part of $K_{\lambda}\times K_{\lambda}$ is defined as follows:
$$V_{\theta}(K_{\lambda}\times K_{\lambda})=\{(x,y)\in K_{\lambda}\times K_{\lambda}: ((x,y)+l_{\theta})\cap (K_{\lambda}\times K_{\lambda})=\{(x,y)\}\}.$$
Let  $(x,y)\in K_{\lambda}\times K_{\lambda}$. Define$$Proj_{\theta}(x,y)=y-x\tan\theta.$$
In other words, we project a point  $(x,y)$ to the $y$-axis in direction $\theta.$
Moreover, we also define the following sets. 
$$Proj_{\theta}(K_{\lambda}\times K_{\lambda})=\{y-x\tan\theta :(x,y)\in K_{\lambda}\times K_{\lambda}\},$$
$$Proj_{\theta}(V_{\theta}(K_{\lambda}\times K_{\lambda}))=\{y-x\tan\theta :(x,y)\in V_{\theta}(K_{\lambda}\times K_{\lambda})\}.$$
 Generally, $Proj_{\theta}(K_{\lambda}\times K_{\lambda}) $ is not an interval.   In what follows, we always assume that  $E=Proj_{\theta}(K_{\lambda}\times K_{\lambda})=[-\tan\theta, 1]$, which is a natural assumption \cite{Barany}. Clearly, $E$ is the attractor of the following IFS,
\begin{eqnarray*}
               g_1(x)&=&\lambda x- (1-\lambda) \tan\theta\\
                g_2(x)&=&\lambda x+(1-\lambda)(1-  \tan\theta)\\
                g_3(x)&=&\lambda x\\
                g_4(x)&=&\lambda x+ 1-\lambda.
\end{eqnarray*}
 For the IFS of $E$, i.e. $\{g_i\}_{i=1}^{4}$, 
define $T_{j}(x):=g_{j}^{-1}(x)$ for $x\in g_j(E)$ and $1\leq j\leq 4$.  
We denote the concatenation $T_{i_{n}}\circ \ldots \circ T_{i_{1}}(x)$ by $T_{i_1\ldots i_n}(x)$. 
Let  $$H_i=g_i(E)\cap g_{i+1}(E), 1\leq i\leq 3,$$ i.e. 
we define  $H=H_1\cup H_2\cup H_3=[a_1,b_1]\cup [a_2,b_2]\cup[a_3,b_3]$. 
The following two propositions are motivated by the results in   open dynamical systems. 
\begin{Proposition}\label{pro1}
Suppose that $Proj_{\theta}(K_{\lambda}\times K_{\lambda})=[-\tan\theta, 1]$.  
For any $[a_i, b_i], 1\leq i\leq 3$, if there are some 
$i_1\cdots i_v, j_1j_2\cdots j_w$ such that 
$$T_{i_1\cdots i_v}(a_i)\in H, T_{j_1j_2\cdots j_w}(b_i)\in H,$$
then 
$$Proj_{\theta}(V_{\theta}(K_{\lambda}\times K_{\lambda}))=\{a\in[-\tan\theta,1]: \sharp\{(x,(\tan\theta) x+a)\cap (K_{\lambda}\times K_{\lambda} ) \}=1\}$$ 
is a graph-directed self-similar sets with the strong separation condition, where 
$\sharp(\cdot)$ denotes the cardinality.  
\end{Proposition}
Analogously, we have the following result. 
\begin{Proposition}\label{pro2}
$Proj_{\theta}(K_{\lambda}\times K_{\lambda})=[-\tan\theta, 1]$.   For any $[a_i, b_i], 1\leq i\leq 3$, if all the possible orbits of $a_i$ and $b_i$ hit finitely many points, then 
apart from a countable set 
$$Proj_{\theta}(V_{\theta}(K_{\lambda}\times K_{\lambda}))=\{a\in[-\tan\theta,1]: \sharp\{(x,(\tan\theta) x+a)\cap (K_{\lambda}\times K_{\lambda} ) \}=1\}$$ 
is a graph-directed self-similar sets with the open set  condition. 
\end{Proposition}
Propositions \ref{pro1} and \ref{pro2} give a sufficient conditition which allows us to calculate the  dimension of the slicing set  $\{a\in[-\tan\theta,1]: \sharp\{(x,(\tan\theta) x+a)\cap (K_{\lambda}\times K_{\lambda} ) \}=1\}$.

The paper is arranged as follows. In section 2, we give the proof of Theorem \ref{Main}. In section 3, we give the proofs of Propositions \ref{pro1} and \ref{pro2}. Finally,  we give some remarks.
\section{Proofs of Main results}
Before we prove Theorem \ref{Main}, we give some definitions, and prove a useful lemma.
 Let $E=[0,1]$. For any $
(i_{1},\cdots ,i_{n})\in \{1,2\}^{n}$, we call $f_{i_{1},\cdots
,i_{n}}([0,1])=(f_{i_{1}}\circ \cdots \circ f_{i_{n}})([0,1])$ a basic interval of
rank $n,$ which has length $\lambda^{n}$. Denote by $E_{n}$ the collection of all these basic intervals of rank $
n$. Suppose $A$ and $B$ are the left and right endpoints of some basic
intervals in $E_{k}$ for some $k\geq 1$, respectively. Denote by $
G_{n}(\subset E_{n})$ the union  of all the basic intervals of rank $n$
which are contained in $[A,B].$ Let $I$ be a basic interval with rank $n$. Define $\widetilde{I}=f_1(I)\cup f_2(I)$. 
\begin{Lemma}
\label{key1} Let $F:U\rightarrow \mathbb{R}$ be a continuous function, where $U\subset \mathbb{R}^2$ is a non-empty open set.
Suppose $A$ and $B$ are the left and right endpoints of some
basic intervals in $G_{k_{0}}$ for some $k_{0}\geq 1$ respectively such that
$[A,B]\times \lbrack A,B]\subset U.$ Then $K\cap \lbrack A,B]=\cap _{n={k_{0}%
}}^{\infty }G_{n}$. Moreover, if for any $n\geq k_{0}$ and any two basic
intervals $I,J\subset G_{n}$,  such that
\begin{equation*}
F(I,J)=F(\widetilde{I},\widetilde{J}),
\end{equation*}%
then $F(K\cap \lbrack A,B],K\cap \lbrack A,B])=F(G_{k_{0}},G_{k_{0}}).$
\end{Lemma}
\begin{proof}
By the construction of $G_{n}$,  i.e. $G_{n+1}\subset
G_{n}$  for any $n\geq k_{0}$,
it follows that
\begin{equation*}
K\cap \lbrack A,B]=\cap _{n=k_{0}}^{\infty }G_{n}.
\end{equation*}
The continuity of $F$ yields that
\begin{equation*}
F(K\cap \lbrack A,B],K\cap \lbrack A,B])=\cap _{n=k_{0}}^{\infty
}F(G_{n},G_{n}).
\end{equation*}
Without loss of generality, we may assume that 
$$G_n=\cup_{1\leq i\leq t_n}I_{n,i} \mbox{ for some } t_n\geq 1,$$ where 
$I_{n,i}$ is a basic interval in $G_n$. 
By the condition in lemma, i.e. 
for any $n\geq k_{0}$ and any two basic
intervals $I,J\subset G_{n}$,  such that
\begin{equation*}
F(I,J)=F(\widetilde{I},\widetilde{J}),
\end{equation*}
 it follows that
\begin{eqnarray*}
F(G_{n},G_{n}) &=&\cup _{1\leq i\leq t_{n}}\cup _{1\leq j\leq
t_{n}}F(I_{n,i},I_{n,j}) \\
&=&\cup _{1\leq i\leq t_{n}}\cup _{1\leq j\leq t_{n}}F(\widetilde{
I_{n,i}},\widetilde{I_{n,j}}) \\
&=&F(\cup _{1\leq i\leq t_{n}}\widetilde{I_{n,i}},\cup _{1\leq j\leq
t_{n}}\widetilde{I_{n,j}}) \\
&=&F(G_{n+1},G_{n+1}).
\end{eqnarray*}
Therefore, $F(K\cap \lbrack A,B],K\cap \lbrack
A,B])=F(G_{k_{0}},G_{k_{0}}).$
\end{proof}
\begin{Lemma}\label{key2}
Let $f(x,y)=\dfrac{x}{y}$, and
$I=[a,a+t],J=[b,b+t]$ be two basic intervals.
If $1/3\leq \lambda<1/2$, and $b\geq a\geq 1-\lambda,$ then
$f(\tilde{I},\tilde{J})=f(I,J).$
\end{Lemma}
\begin{proof}
Note that
$$\tilde{I}=[a,a+\lambda t]\cup \lbrack a+t-\lambda t,a+t],\tilde{J}=[b,b+\lambda t]\cup \lbrack b+t-\lambda t,b+t].$$
Therefore,
$$f(\tilde{I},\tilde{J})=J_{1}\cup J_{2}\cup J_{3}\cup J_{4},$$
where
\begin{eqnarray*}
\begin{split}
&J_1=\left[\frac{a}{b+t},\frac{a+\lambda t}{b+t-\lambda t}%
\right]=:[r_{1},s_{1}]  \\
&J_2=\left[\frac{a}{b+\lambda t},\frac{a+\lambda t}{b}%
\right]=:[r_{2},s_{2}]\\
&J_3=\left[\frac{a+t-\lambda t}{b+t},\frac{a+t}{b+t-\lambda t}%
\right]=:[r_{3},s_{3}]\\
&J_4= \left[\frac{a+t-\lambda t}{b+\lambda t},\frac{a+t}{b}%
\right]=:[r_{4},s_{4}].
\end{split}
\end{eqnarray*}
Note that $f(I,J)=[r_1, s_4]$. In the following, we verify that $f(I,J)=J_{1}\cup J_{2}\cup J_{3}\cup J_{4}$

Since $b\geq a\geq 1-\lambda$ and $\lambda \geq \frac{1}{3}$, we have
$$r_{3}-r_{2}=\frac{a+t-\lambda t}{b+t}-\frac{a}{b+\lambda t}=
\frac{t(1-\lambda)(b-a+t\lambda)}{\left( b+t\lambda \right) \left( b+t\right) }\geq 0.$$

Now it suffices to check that
$$
                s_1-r_2\geq0, \;\;
                s_2-r_3\geq0\;\;\textrm{and}\;\;
                s_3-r_4\geq0.
$$
We have
$$s_1-r_2=\frac{a+\lambda t}{b+t-\lambda t}-\frac{a}{b+\lambda t}=
\frac{t( 2a\lambda
-a+b\lambda +t\lambda ^{2})}{(b+t-\lambda t)(b+\lambda t)}\geq
\frac{t(a(3\lambda -1) +t\lambda ^{2})}{(b+t-\lambda t)(b+\lambda t)}
 \geq 0,$$
and
 \begin{equation*}
 \begin{split}
 s_2-r_3&=\frac{a+\lambda t}{b}-\frac{a+t-\lambda t}{b+t}=
\frac{t( a+(2\lambda -1)b +t\lambda)}{b(b+t)} \\
&\geq \frac{t( b(1-\lambda )+(2\lambda -1)b +t\lambda)}{b(b+t)}
\geq 0.
\end{split}
\end{equation*}
Finally,
 \begin{equation*}
 \begin{split}
s_3-r_4 =\frac{a+t}{b+t-\lambda t}-\frac{a+t-\lambda t}{b+\lambda t}=
\frac{t( -a-t+2a\lambda
+b\lambda +3t\lambda -t\lambda ^{2})}{(b+t-\lambda t)(b+\lambda t) }.
\end{split}
\end{equation*}
If $b\neq a$,  then $b>a+t$. Therefore, we have
\begin{eqnarray*}
-a-t+2a\lambda
+b\lambda +3t\lambda -t\lambda ^{2}&\geq& -a+2a\lambda+(a+t)\lambda+t(3\lambda-1-\lambda^2)\\
&\geq
& a(3\lambda-1)+t(4\lambda-1-\lambda^2)\geq 0.
\end{eqnarray*}
which leads to $s_3-r_4\geq 0$. However, if  $a=b$, then
$$
s_2-r_4=\frac{a+\lambda t}{a}-\frac{a+t-\lambda t}{a+\lambda t}=\dfrac{\lambda^2t^2+at(3\lambda -1)}{a(a+\lambda t)}\geq 0.
$$
Thus, we finish checking that $f(I,J)=J_{1}\cup J_{2}\cup J_{3}\cup J_{4}=[r_1, s_4]=\left[\dfrac{a}{b+t},\dfrac{a+t}{b}\right]$.
\end{proof}

\begin{Lemma}\label{lem1}
We have
\begin{equation}
  \frac{K_{\lambda }}{
K_{\lambda }\setminus \{0\}}=\left \{
\begin{array}{ll}
[0,\infty), &  \textrm{when}\;\; \frac{3-\sqrt{5}}{2}\leq \lambda<1/2\\
 \bigcup_{k=-\infty}^{+\infty}\lambda^k\left[1-\lambda, \frac{1}{1-\lambda}\right]\cup\{0\} & \textrm{when}\;\;
 1/3\leq \lambda<\frac{3-\sqrt{5}}{2}.
 \end{array}
 \right.
\end{equation}
\end{Lemma}
\begin{proof}
From Lemmas \ref{key1} and \ref{key2} it follows that   if  $\lambda \geq \frac{1}{3}$
$$
\dfrac{f_2(K_{\lambda})}{f_2(K_{\lambda})}=\left[1-\lambda ,\frac{1}{1-\lambda}\right].
$$
 Each $x\in K_\lambda $ can be uniquely represented as
$$
x=\sum _{n=1}^\infty x_n\lambda ^n\;\;\textrm{with}\;\; x_n\in \{0, 1-\lambda \}.
$$
Note that $x\in f_2(K_\lambda )$ if and only if $x_1=1-\lambda $. Thus each $x\in K_\lambda \setminus \{0\}$ is of form
$$
x=\lambda ^mx^*\;\;\textrm{with}\;\; m\in \{0,1,2,\cdots \}, x^*\in f_2(K_\lambda ).
$$
Thus for any two $x=\lambda ^mx^*, y=\lambda ^ny^*\in K_\lambda \setminus \{0\}$ with $x^*, y^*\in f_2(K_\lambda )$ one has
$$
\frac{x}{y}=\lambda ^{m-n}\cdot \frac{x^*}{y^*}\in \lambda ^{m-n}\left [1-\lambda , \frac{1}{1-\lambda }\right ].
$$
Thus
$$
\frac{K_{\lambda }}{
K_{\lambda }\setminus \{0\}}=\{0\}\cup \bigcup_{k=-\infty}^{\infty}\lambda^k\left[1-\lambda, \dfrac{1}{1-\lambda}\right].
$$
It is easy to check that $\{0\}\cup \bigcup_{k=-\infty}^{\infty}\lambda^k\left[1-\lambda, \dfrac{1}{1-\lambda}\right]=[0, +\infty )$ when $\dfrac{3-\sqrt{5}}{2}\leq \lambda<1/2$, and intervals
$\lambda^k\left[1-\lambda, \dfrac{1}{1-\lambda}\right]$ are pairwise disjoint when $1/3\leq \lambda<\dfrac{3-\sqrt{5}}{2}$.
\end{proof}

Pourbarat \cite{Pourbarat}, making use of the thickness of the Cantor sets,  proved the following result.
\setcounter{Theorem}{4}
\begin{Theorem}\label{pou}
If $\dfrac{\lambda^2}{(1-2\lambda)^2}>\lambda$, then $ \dfrac{K_{\lambda }}{
K_{\lambda}\setminus \{0\}}$ contains an interior point.
\end{Theorem}
\setcounter{Lemma}{5}
\begin{Lemma}\label{lem2}
If $\dfrac{1}{4}<\lambda<1/3$, then
$ \dfrac{K_{\lambda }}{
K_{\lambda}\setminus \{0\}}$ contains an   interior point.
\end{Lemma}
\begin{proof}
If $\dfrac{1}{4}<\lambda<1/3$, then $\dfrac{\lambda^2}{(1-2\lambda)^2}>\lambda$. Therefore, $ \dfrac{K_{\lambda }}{
K_{\lambda}\setminus \{0\}}$ contains an interior point by Theorem \ref{pou}.
\end{proof}

\begin{Lemma}\label{lem3}
If $0<\lambda<\dfrac{3-\sqrt{5}}{2}$, then $V$ has an interior point.
\end{Lemma}
\begin{proof}
Note that $f_2(K_{\lambda})\subset[1-\lambda,1]$. Thus, by the argument in Lemma \ref{lem1}, we have
$$
\dfrac{K_{\lambda}}{K_{\lambda}\setminus \{0\}}\subseteq \{0\} \cup\bigcup_{k=-\infty}^{\infty}  \lambda^k\left[1-\lambda, \dfrac{1}{1-\lambda}\right].
$$
Note that the intervals $\left[\lambda^k(1-\lambda ), \dfrac{\lambda^k}{1-\lambda}\right]$ for $k\in \mathbb Z$ are pairwise disjoint when $0<\lambda<\dfrac{3-\sqrt{5}}{2}$.
Therefore,  $V$ has an interior point by (\ref{setV}).
\end{proof}

\begin{Lemma}\label{lem4}
If $0<\lambda<\dfrac{1}{4}$, then $\dfrac{K_{\lambda}}{K_{\lambda}\setminus \{0\}}$ has Lebesgue measure zero.
\end{Lemma}
\begin{proof}
If $0<\lambda<\dfrac{1}{4}$, then
$$\dim_{H}K_{\lambda}+\dim_{H}K_{\lambda}<1.$$
We note that for any $X,\,Y\subseteq \mathbb{R}$, we have
$Y-X=\Pi_{\frac{\pi}{4}}(X\times Y)$, where $\Pi_{\frac{\pi}{4}}(X\times
Y)$ denotes the projection of $X\times Y$ on the $y$ axis along lines having $\frac{\pi}{4}$ angle with the $x$ axis.
Therefore,
\begin{eqnarray*}
\begin{split}
\dim_{H}\dfrac{K_{\lambda}}{K_{\lambda}\setminus \{0\}}&=\dim_{H}\dfrac{K_{\lambda}\setminus \{0\}}{K_{\lambda}\setminus \{0\}}=\dim_{H}(\ln (K_{\lambda}\setminus \{0\})-\ln (K_{\lambda}\setminus \{0\})\\
\;&\leq \dim_{H}(\ln (K_{\lambda}\setminus \{0\})\times \ln (K_{\lambda}\setminus \{0\})\\
\;&\leq
\dim_{H}(\ln K_{\lambda}\setminus \{0\})+\dim_{P}(\ln K_{\lambda}\setminus \{0\})\\
\:&=2\dim_{H}( K_{\lambda}\setminus \{0\})<1,
\end{split}
\end{eqnarray*}
where we use the fact that $\dim _HE=\dim _H\ln E$ and $\dim _PE=\dim _P\ln E$.
for a bounded set $E\subset (0, +\infty )$. Thus, $\dfrac{K_{\lambda}}{K_{\lambda}\setminus \{0\}}$ has Lebesgue measure zero.
\end{proof}

We will use a result given by Simon and Solomyak \cite{SimonSolomyak}.
\setcounter{Theorem}{9}
\begin{Theorem} \label{Simon and Solomyak}
 Let $\Lambda$ be a self-similar  $1$-set in $\mathbb{R}^2$ with the open set condition, which is not on a line. Then
$$m(P_{(0,0)}(\Lambda\setminus\{(0,0)\}))=0,$$
where $$P_{(0,0)}:\mathbb{R}^2\setminus{(0,0)}\to S^1, P_{(0,0)}(\vec{x})=\dfrac{\vec{x}}{|\vec{x}|}.$$
\end{Theorem}

\begin{Lemma}\label{lem5}
$\dfrac{K_{1/4}}{K_{1/4}\setminus \{0\}}$ has Lebesgue measure zero.
\end{Lemma}
\begin{proof}
Note that when $\lambda=1/4$, $\Lambda=K_{\lambda}\times K_{\lambda}$ is a self-similar set with the following IFS
$$g_1(x,y)=\left(\dfrac{x}{4}, \dfrac{y}{4}\right), g_2(x,y)=\left(\dfrac{x+3}{4}, \dfrac{y}{4}\right)$$
$$ g_3(x,y)=\left(\dfrac{x+3}{4}, \dfrac{y+3}{4}\right), g_4(x,y)=\left(\dfrac{x}{4}, \dfrac{y+3}{4}\right).$$
Clearly, the above IFS satisfies the open set condition. Therefore, the Hausdorff dimension of $\Lambda$ is 1, and  $0<\mathcal{H}^{1}(\Lambda)<\infty$.
Let $$\Gamma=\left\{\dfrac{(x,y)}{\sqrt{x^2+y^2}}\in S^1:(x,y)\in K_{\lambda}\times K_{\lambda}\setminus\{(0,0)\}\right\}=P_{(0,0)}(\Lambda\setminus\{(0,0)\})$$
The Lebesgue measure of $\Gamma$ is $0$ due to Theorem \ref{Simon and Solomyak}.
Let
$$\Gamma_1=\left\{\dfrac{(x,y)}{\sqrt{x^2+y^2}}\in S^1:(x,y)\in K_{\lambda}\times K_{\lambda}\setminus\{(0,0)\}, x\neq 0\right\}.$$
Clearly, $m(\Gamma_1)=m(\Gamma)=0$.
  The metric on $\Gamma_1$, denoted by $d_1$, is the arc metric. It is well known that on $S^1$, the arc metric is equivalent to the Euclidean metric.
Let $$\Gamma_2=\left\{\arctan \dfrac{y}{x}:(x,y)\in K_{\lambda}\times K_{\lambda}\setminus\{(0,0)\}, x\neq 0 \right\}.$$ The  metric on  $\Gamma_2$ the Euclidean metric (we denote it by $d_2$).
We define the the  map
$$\phi: \Gamma_1\to \Gamma_2,$$ by $$ \phi\left(\dfrac{(x,y)}{\sqrt{x^2+y^2}}\right)=\arctan \dfrac{y}{x}.$$
The map $\phi$ is indeed mapping a point on $S^1$ into its associated  polar angle in the polar coordinate system. Therefore, we may define
$\phi$ in another way as follows:
define $$\phi: \Gamma_1\to \Gamma_2, \;\;\phi(\vec{a})=\theta_{\vec{a}}$$
Clearly, $\phi$ is well-defined, and it is a bijection. Moreover, we shall prove that $\phi$ is a Lipschitz map,  i.e. there exists some constant $L>0$ such that
$$d_2(\phi(\vec{a}), \phi(\vec{b}))\leq L d_1(\vec{a},\vec{b}).$$
Note that $d_2(\phi(\vec{a}), \phi(\vec{b}))=d_2(\theta_{\vec{a}},\theta_{\vec{b}})$, and that
$$d_1(\vec{a},\vec{b})=d_2(\theta_{\vec{a}}\cdot 1, \theta_{\vec{b}}\cdot 1)=d_2(\theta_{\vec{a}}, \theta_{\vec{b}}).$$
Now,  $m(\Gamma_2)=0$ follows from  $\phi(\Gamma_1)=\Gamma_2$,  $m(\Gamma_1)=0$, and $\phi$ is Lipschitz.
Theorefore, $m\left(\dfrac{K_{1/4}}{K_{1/4}\setminus \{0\}}\right)=0.$
\end{proof}
The following is from
B{\'a}r{\'a}ny  \cite{BABA1}.
\begin{Theorem}\label{BBB}
Let $\Lambda$ be an arbitrary self-similar set in $\mathbb{R}^2$ not contain in any line. Suppose that
$g:\mathbb{R}^2\to \mathbb{R}$ is a $C^2$ map such that
$$(g_x)^2+(g_y)^2\neq 0, (g_{xx}g_y-g_{xy}g_x)^2+(g_{xy}g_y-g_{yy}g_x)^2\neq 0$$
for any $(x,y)\in \Lambda$. Then
$$\dim_{H}g(\Lambda)=\min\{1,\dim_{H}(\Lambda)\}.$$
\end{Theorem}
\begin{Lemma}\label{dimm}
When $0< \lambda\leq \dfrac{1}{4}$, $\dim _H([0,+\infty )\setminus V)=\dfrac{\log 4}{-\log \lambda}$.
\end{Lemma}
\begin{proof}
By the argument in Lemma \ref{lem1}, we have
$$ \frac{K_{\lambda }}{
K_{\lambda }\setminus \{0\}}=\bigcup_{k=-\infty}^{\infty}\lambda^k\dfrac{f_2(K_{\lambda})}{f_2(K_{\lambda})}\cup\{0\}.
$$
Thus
$$
\dim_{H}\frac{K_{\lambda }}{
K_{\lambda }\setminus \{0\}}=\dim_{H}\dfrac{f_2(K_{\lambda})}{f_2(K_{\lambda})}.$$
Clearly, $\Lambda=f_2(K_{\lambda})\times f_2(K_{\lambda})$ is a two-dimensional self-similar set which is not contained in  any line.  Let $g(x,y)=\dfrac{x}{y}$, then
$$
(g_x)^2+(g_y)^2\neq 0, (g_{xx}g_y-g_{xy}g_x)^2+(g_{xy}g_y-g_{yy}g_x)^2\neq 0
$$
for any $(x,y)\in \Lambda$. Therefore, in terms of Theorem \ref{BBB},
$$\dim g(\Lambda)=\dim_{H}\dfrac{f_2(K_{\lambda})}{f_2(K_{\lambda})}=\min\{\dim_{H}(f_2(K_{\lambda})\times f_2(K_{\lambda})),1\}=\min\{2\dim_{H}(K_{\lambda}),1\}.$$
Hence, if  $0 < \lambda\leq 1/4$, then
$$\dim_{H}\dfrac{f_2(K_{\lambda})}{f_2(K_{\lambda})}=2\dim_{H}(K)=\dfrac{\log 4}{-\log \lambda}.$$
\end{proof}

\begin{proof}[Proof of Theorem \ref{Main}]
Theorem \ref{Main} (1) and (2) follows from Lemmas \ref{lem1}. Theorem \ref{Main} (3) follows from  \ref{lem2}, \ref{lem3}, \ref{lem4},  \ref{lem5} and \ref{dimm}.
\end{proof}

\section{Visible sets, slicing sets and open dynamical systems}
In this section, we give the proofs of Propositions \ref{pro1} and \ref{pro2}. 
Define 
$$P_{ij}=Proj_{\theta}(f_i([0,1])\cap f_j([0,1]))=\{y-x\tan\theta:(x,y)\in f_i([0,1])\cap f_j([0,1])\},1\leq i,j\leq 2.$$
It is easy to see that the length of $P_{ij}$ is $\lambda(1+\tan \theta).$   In what follows, we always assume that  $E=Proj_{\theta}(K_{\lambda}\times K_{\lambda})=[-\tan\theta, 1]$. Clearly, in terms of $P_{ij}$,  $E$ is the attractor of the following IFS,
\begin{eqnarray*}
               g_1(x)&=&\lambda x- (1-\lambda) \tan\theta\\
                g_2(x)&=&\lambda x+(1-\lambda)(1-  \tan\theta)\\
                g_3(x)&=&\lambda x\\
                g_4(x)&=&\lambda x+ 1-\lambda.
\end{eqnarray*}
In other words, $$E=\cup_{i=1}^{4}g_i(E)=[-\tan\theta,1].$$ 
For any  $x \in E$, there exists a sequence
$(i_n)_{n=1}^{\infty}\in\{1,\ldots,4\}^{\mathbb{N}}$ such that
\[x=\lim_{n\to \infty}g_{i_1}\circ \cdots\circ g_{i_n}(0).\] We call such a sequence a coding of $x$.
Usually, the coding of $x$ is not unique. If $x$ has a unique coding, then we call $x$ a univoque point. Write $U_1$ for all the univoque points of $E$ with respect to the IFS $\{g_i\}_{i=1}^{4}.$

The following result is proved in \cite[Lemma 2.1]{SKK} which states that any self-similar set can be regarded as a topological dynamical system. For the IFS of $E$, i.e. $\{g_i\}_{i=1}^{4}$, 
define $T_{j}(x):=g_{j}^{-1}(x)$ for $x\in g_j(E)$ and $1\leq j\leq 4$.  
We denote the concatenation $T_{i_{n}}\circ \ldots \circ T_{i_{1}}(x)$ by $T_{i_1\ldots i_n}(x)$. \begin{Lemma}
\label{Dynamical lemma}
Let $x\in K.$ Then $(i_{n})_{n=1}^{\infty}\in\{1,\ldots,m\}^{\mathbb{N}}$ is a coding for $x$ if and only if $T_{i_{1}\ldots i_{n}}(x)\in K$ for all $n\in\mathbb{N}.$
\end{Lemma}
Motivated by Lemma \ref{Dynamical lemma}, we may define the orbits of the points of $K$.
\setcounter{Definition}{1}
\begin{Definition}\label{Orbit set}
Let $x\in K$ with  a coding  $(i_n)_{n=1}^{\infty}$, we call  the set $$\{T_{i_{1}\ldots i_{n}}(x): n\geq 0\}$$ an  orbit set of  $x$, where $T_{i_0}(x)=x.$
\end{Definition}
It is easy to see that for different codings, the orbits of $x$ may be distinct.

The set $Proj_{\theta}(V_{\theta}(K_{\lambda}\times K_{\lambda}))$ can be viewed as a slicing set, i.e. 
$$Proj_{\theta}(V_{\theta}(K_{\lambda}\times K_{\lambda}))=\{a\in[-\tan\theta,1]: \sharp\{(x,(\tan\theta) x+a)\cap (K_{\lambda}\times K_{\lambda} ) \}=1\}, $$
where $\sharp(\cdot)$ denotes the cardinality. 
The following lemma is trivial. 
\begin{Lemma}
Suppose that $Proj_{\theta}(K_{\lambda}\times K_{\lambda})=[-\tan\theta, 1]$. Then 
$$Proj_{\theta}(V_{\theta}(K_{\lambda}\times K_{\lambda}))$$ is exactly the univoque set of $E$ under the IFS $\{g_i\}_{i=1}^{4},$ i.e. 
$$Proj_{\theta}(V_{\theta}(K_{\lambda}\times K_{\lambda}))=\{a\in[-\tan\theta,1]: \sharp\{(x,(\tan\theta) x+a)\cap (K_{\lambda}\times K_{\lambda} ) \}=1\}=U_1.$$
\end{Lemma}
\begin{proof}
The proof is trivial. We leave it to  readers.
\end{proof}
If $\theta\in [\pi/4,\pi/2)$, let  $H_i=g_i(E)\cap g_{i+1}(E), 1\leq i\leq 3$, i.e. 
\begin{eqnarray*}
H_1&=&[1-\lambda-\tan\theta, \lambda-(1-\lambda)\tan\theta]\\
H_2&=&[-\lambda\tan\theta, 1-(1-\lambda)\tan\theta]\\
H_3 &=& [1-\lambda-\lambda \tan\theta,\lambda].
\end{eqnarray*}
Let   $H=H_1\cup H_2\cup H_3=[a_1,b_1]\cup [a_2,b_2]\cup[a_3,b_3]$. 
Similarly, if $\theta\in (0,\pi/4)$, then we can also define the $H_i=[a_i,b_i], 1\leq i\leq 3.$ 
Now the following result is a corollary of  the main result of \cite{SKK}. 
\begin{Proposition}
Suppose that $Proj_{\theta}(K_{\lambda}\times K_{\lambda})=[-\tan\theta, 1]$.  
For any $[a_i, b_i], 1\leq i\leq 3$, if there are some 
$i_1\cdots i_v, j_1j_2\cdots j_w$ such that 
$$T_{i_1\cdots i_v}(a_i)\in H, T_{j_1j_2\cdots j_w}(b_i)\in H,$$
then 
$$\{a\in[-\tan\theta,1]: \sharp\{(x,(\tan\theta) x+a)\cap (K_{\lambda}\times K_{\lambda} ) \}=1\}$$ 
is a graph-directed self-similar sets with the strong separation condition. 
\end{Proposition}
The following  is a corollary of  the main result of \cite{KarmaKan}. 
\begin{Proposition}
$Proj_{\theta}(K_{\lambda}\times K_{\lambda})=[-\tan\theta, 1]$.   For any $[a_i, b_i], 1\leq i\leq 3$, if all the possible orbits of $a_i$ and $b_i$ hit finitely many points, then 
apart from a countable set 
$$\{a\in[-\tan\theta,1]: \sharp\{(x,(\tan\theta) x+a)\cap (K_{\lambda}\times K_{\lambda} ) \}=1\}$$ 
is a graph-directed self-similar sets with the open set condition. 
\end{Proposition}

\section{Some remarks}
The main idea of this paper is to establish a connection  between  the visible problem and  arithmetic on the fractal sets.
Our idea can be implemented for other overlapping self-similar sets.
Similar results can be obtained if we replace the line $y=\alpha x$ in the definition of $V$,  by some
  parabolic curves or hyperbolic curves.  Nevertheless, for these cases, the analysis can be difficult. We shall discuss these problems in another paper.

  \section*{Acknowledgements}
The work is supported by National Natural Science Foundation of China (Nos.11701302,

11671147). The work is also supported by K.C. Wong Magna Fund in Ningbo University.


\end{document}